\documentclass[12pt]{article}
\usepackage[top=2cm, bottom=2cm, left=3cm , right=3cm]{geometry} 
\usepackage[utf8]{inputenc}
\usepackage[T1]{fontenc}
\usepackage{amssymb}
\usepackage{amsmath}
\usepackage{dsfont}
\usepackage{mathtools}
\usepackage{amsfonts}
\usepackage[greek, english]{babel}
\usepackage{amsthm}
\usepackage{tikz}
\usepackage{hyperref}
\usepackage{enumitem}  
\usepackage{marginnote}
\usepackage[symbol]{footmisc}

\usepackage[draft,inline,nomargin,index]{fixme}

\fxsetup{theme=color,mode=multiuser}
\FXRegisterAuthor{cj}{acj}{\color{green}CJ}
\FXRegisterAuthor{mj}{amj}{\color{red}MJ}
\FXRegisterAuthor{rb}{arb}{\color{blue}RB}

\usepackage{setspace}
\usepackage{tikz-cd} 

\addto\captionsfrench{%
  }

\newenvironment{cproof}{\begin{proof}[Proof of the 
		claim]}{\end{proof}}
\newtheoremstyle{capitale}    
    {\topsep}                    
    {\topsep}                    
    {\itshape}                   
    {}                           
    {\scshape}                   
    {. ---}                          
    {.5em}                       
    {} 
    
\newtheoremstyle{remark}    
    {\topsep}                    
    {\topsep}                    
    {}                   
    {}                           
    {\scshape}                   
    {. ---}                          
    {.5em}                       
    {} 

\theoremstyle{capitale}
\newtheorem{theorem}{Theorem}[section]
\newtheorem{lemma}[theorem]{Lemma}
\newtheorem*{claim*}{Claim}

\newtheorem*{lemme*}{Lemme}
\newtheorem{corollary}[theorem]{Corollary}
\newtheorem{definition}[theorem]{Definition}
\newtheorem*{definition*}{Definition}
\newtheorem{fact}{Fact}

\theoremstyle{remark}
\newtheorem{remark}[theorem]{Remark}

\newtheorem{example}[theorem]{Example}

\usepackage[backend=biber, style=numeric, maxnames=50]{biblatex}
\usepackage{csquotes}
\usepackage{todonotes}
\DefineBibliographyStrings{english}{in = {}}

\newcommand{\BB}{\mathcal{B}}

\newcommand{\FF}{\mathcal{F}}

\newcommand{\HH}{\mathcal{H}}

\newcommand{\KK}{\mathcal{K}}

\newcommand{\UU}{\mathcal{U}}

\newcommand{\Isom}{\mathrm{Isom}}
\newcommand{\Ind}{\mathrm{Ind}}

\newcommand{\EE}{\mathcal{E}}
\renewcommand{\L}{\mathrm{L}}

\newcommand{\N}{\mathbb{N}}
\newcommand{\Z}{\mathbb{Z}}
\newcommand{\Q}{\mathbb{Q}}
\newcommand{\U}{\mathbb{U}}
\newcommand{\R}{\mathbb{R}}
\newcommand{\C}{\mathbb{C}}

\newcommand{\Sym}{\mathrm{Sym}}

\newcommand{\inv}{^{-1}}

\newcommand{\defin}[1]{\textbf{\textit{#1}}}

\newcommand{\Aut}{\mathrm{Aut}}
\newcommand{\Sub}{\mathrm{Sub}}

\renewcommand{\P}{\mathbb{P}}

\newcommand\Nm[2]{{#1}_{\{#2\}}}

\newcommand\conv[2]{\operatorname{Conv}_{#1}(#2)}

\newcommand{\indep}{%
\mathrel{\reflectbox{\rotatebox[origin=c]{90}{$\models$}}}}

\newcommand{\supp}{\mathrm{supp}}

\newcommand{\Addresses}{{
			\bigskip
			\footnotesize
			\noindent R.~Barritault, \textsc{Université Claude Bernard Lyon 1, Institut Camille Jordan, \\Lyon, France}\par\nopagebreak\noindent
			\textit{E-mail address: }\texttt{barritault@math.univ-lyon1.fr}

                \medskip

			\noindent C.~Jahel, \textsc{Institut fur Algebra, Technische Universität Dresden, \\Dresden, Germany}\par\nopagebreak\noindent
			\textit{E-mail address: }\texttt{colin.jahel@tu-dresden.de}
			
			\medskip
			
			\noindent M.~Joseph, \textsc{Université Paris-Saclay, CNRS, Laboratoire de mathématiques d’Orsay, \\Orsay, France}\par\nopagebreak\noindent
			\textit{E-mail address: }\texttt{matthieu.joseph@universite-paris-saclay.fr}
	
			\medskip

	}}

\title{Unitary representations of the isometry groups of Urysohn spaces}

\author{Rémi BARRITAULT, Colin JAHEL, Matthieu JOSEPH}
\date{}

\begin{document}
\maketitle

\begin{abstract}
    We obtain a complete classification of the continuous unitary representations of the isometry group of the rational Urysohn space $\Q\U$. As a consequence, we show that $\Isom(\Q\U)$ has property (T). We also derive several ergodic theoretic consequences from this classification: $(i)$ every probability measure-preserving action of $\Isom(\Q\U)$ is either essentially free or essentially transitive, $(ii)$ every ergodic $\Isom(\Q\U)$-invariant probability measure on $[0,1]^{\Q\U}$ is a product measure. We obtain the same results for isometry groups of variations of $\Q\U$, such as the rational Urysohn sphere $\Q\U_1$, the integral Urysohn space $\Z\U$, etc.
\end{abstract}


\tableofcontents

\newpage
\section{Introduction}

Let $\Q\U$ be the rational Urysohn space, which is the unique countable metric space satisfying the following two conditions:

\begin{enumerate}
    \item (\emph{universality}) every countable metric space with rational distances embeds isometrically into $\Q\U$,
    \item (\emph{ultrahomogeneity}) every isometry between finite subspaces of $\Q\U$ extends to an isometry of $\Q\U$.
\end{enumerate}

The group $\Isom(\Q\U)$ of all isometries of $\Q\U$ onto itself is a Polish group when equipped with the topology of pointwise convergence (with $\Q\U$ viewed as a discrete topological space). It has been the object of intensive study over the last three decades \cite{CameronVershik}, \cite{Hubicka}, \cite{EGLMM}, \cite{FLMMM}, \cite{Melleray}, \cite{TentZiegler}, \cite{TentZieglerBounded}, \emph{etc}.

A family of similar metric spaces can be defined as follows. A distance set is either a countable additive subsemigroup of the positive reals that contains $0$, or the intersection $\mathcal{S}\cap [0,r]$ of such a semigroup $\mathcal{S}$ and a bounded interval with $r\in\mathcal{S}$. To each  distance set corresponds a Urysohn $\Delta$-metric space $\U_\Delta$: it is the unique countable metric space which is ultrahomogeneous and universal among countable metric spaces with distances in $\Delta$. Our main result is a classification of the unitary representations of the Polish groups $\Isom(\U_\Delta)$.

\begin{theorem}[see Theorem \ref{thm.classification.unirep.Isom(UDelta)}]\label{thmintro.unirep.Isom(UQ)}
    Let $\Delta\subseteq\R_+$ be a countable distance set and $G=\Isom(\U_\Delta)$. Every continuous irreducible unitary representation of $G$ is induced from an irreducible representation of the setwise stabilizer $G_{\{A\}}$ for some finite  $A\subseteq\U_\Delta$, which is trivial on the pointwise stabilizer $G_A$. Moreover, every continuous unitary representation of $G$ is a direct sum of irreducible ones. 
\end{theorem}

In particular, every irreducible representation of $G=\Isom(\U_\Delta)$ is a subrepresentation of a quasi-regular representation $\ell^2(G/V)$ for some open subgroup $V\leqslant G$ which in turn is a subrepresentation of $\ell^2(\U_\Delta^n)$ for some $n\in \N$.

As a direct corollary, we obtain the following result. We refer to \cite[Sec.~6]{BdlH} for an introduction to topological groups of type I.

\begin{corollary}
    The group $\Isom(\U_\Delta)$ is of type I for every distance set $\Delta\subseteq\R_+$.
\end{corollary}

Let us mention that a classification of the unitary representations for the isometry group of \emph{ultrametric} variants of $\Q\U$ has recently been achieved by Neretin \cite{Neretin}. Even though the results obtained there have the same flavor, the techniques of proofs are different and Neretin's results do not overlap with ours. Using our description of the unitary representations of $\Isom(\U_\Delta)$ together with some techniques developed by Bekka \cite{Bekka} and Evans-Tsankov \cite{EvansTsankov}, we prove in Section \ref{sec.property(T)} that $\Isom(\U_\Delta)$ has property (T). 

\begin{theorem}\label{thmintro.property(T)}
For every distance set $\Delta\subseteq\R_+$, the Polish group
$\Isom(\U_\Delta)$ has property (T). More precisely, for every integer $m\geq 1$, one can find a Kazhdan set $Q_m$ of cardinality $m$ such that
\[\sqrt{2-\frac{2\sqrt{2m-1}}{m}}\] 
is the optimal Kazhdan constant for $Q_m$.
\end{theorem}

The group $\Isom(\U_\Delta)$ being topologically simple (Theorem 7.3 of \cite{EGLMM}), our work provides new examples of topologically simple Polish groups which have property (T) and are 
\begin{itemize}
    \item locally bounded but not locally Roelcke precompact, e.g.\ $\Isom(\Q\U)$,
    \item locally Roelcke precompact but not Roelcke precompact, e.g.\ $\Isom(\Z\U)$,
    \item coarsely bounded but not Roelcke precompact, e.g.\ $\Isom(\Q\U_1)$. 
\end{itemize} 
We refer to Lemma \ref{lem.coarse.geometry} for large scale geometric properties of the groups $\Isom(\U_\Delta)$. Let us mention that if $\Delta$ is a finite distance set, the group $\Isom(\U_\Delta)$ is oligomorphic. In that case, Tsankov proved in \cite{Tsankov} that $\Isom(\U_\Delta)$ has property (T). He derived this result as a consequence of the classification of the continuous unitary representations for oligomorphic groups. Generalizing Tsankov's results, Property (T) for the very general class of Roelcke precompact groups was later established by Ibarlucía \cite{Ibarlucia}.\\

In Section \ref{sec.Urysohnbig}, we settle some results around representations of the isometry groups of the Urysohn space $\U$ that are probably well-known but written nowhere in the literature. Those results answer several questions asked by Pestov in \cite{Pestov}. More precisely, we use our classification of the unitary representations of $\Isom(\Q\U)$ to prove that $\Isom(\U)$ admits no non-trivial unitary representation. In fact, much more is true: $\Isom(\U)$ has no non-trivial representation by isometry on reflexive Banach spaces (Corollary \ref{cor.Isom(U)}). In particular, $\Isom(\U)$ has property (T).\\

Let us explain the argument that we develop to prove Theorem \ref{thmintro.unirep.Isom(UQ)}. Let $\Omega$ be a countably infinite set and let $\Sym(\Omega)$ be the group of \emph{all} permutations of $\Omega$. It is endowed with the topology induced from the product one on $\Omega^\Omega$ ($\Omega$ is equipped with the discrete topology). A closed  subgroup $G$ of $\Sym(\Omega)$ is called a \defin{closed permutation group}. A neighborhood basis of the identity consisting of clopen subgroups is given by \defin{pointwise stabilizers} of finite subsets, which are defined for $A\subseteq\Omega$ finite by $G_A\coloneqq\{g\in G\colon \forall a\in A, g(a)=a\}$. The \defin{setwise stabilizer} of a finite set $A\subseteq\Omega$ is the subgroup $G_{\{A\}}\coloneqq\{g\in G\colon \forall g\in G, g(A)=A\}$. A closed subgroup $G\leq \Sym(\Omega)$ has \defin{no algebraicity} if for every finite subset $A\subseteq\Omega$, the action $G_A\curvearrowright\Omega\setminus A$ has no finite orbit. Fix a closed permutation group $G\leq \Sym(\Omega)$ and a continuous unitary representation $\pi \colon G\to\UU(\HH)$. Given a finite subset $A \subseteq \Omega$, we let $\HH_A\subseteq\HH$ be the subspace of $G_A$-invariant vectors. The representation $\pi$ is \defin{dissociated} if for all finite subsets $A,B\subseteq\Omega$, the subspaces $\HH_A$ and $\HH_B$ are orthogonal conditionally on $\HH_{A\cap B}$. Given three subspaces $\HH_1,\HH_2,\HH_3$ of a Hilbert space $\HH$ satisfying $\HH_2\subseteq\HH_1\cap\HH_3$, we say that $\HH_1$ and $\HH_3$ are orthogonal conditionally on $\HH_2$ if the subspaces $\HH_1\cap(\HH_2)^\bot$ and $\HH_3\cap(\HH_2)^\bot$ are orthogonal. For closed permutation groups with no algebraicity, we classify dissociated unitary representations.

\begin{theorem}[see Theorem \ref{thm.classification.dissociated}]\label{thmintro.classification.dissociated}
    Let $G\leq\Sym(\Omega)$ be a closed subgroup with no algebraicity. Then every dissociated unitary representation $\pi\colon G\to\UU(\HH)$ is isomorphic to a direct sum $\bigoplus_{i\in I}\pi_i$ of irreducible representations, where for every $i\in I$, there exists a finite subset $A_i\subseteq\Omega$ such that $\pi_i$ is induced from an irreducible representation $\sigma_i$ of the setwise stabilizer $G_{\{A_i\}}$ which factors through the finite group $G_{\{A_i\}}/G_{A_i}$. 
\end{theorem}

In order to obtain Theorem \ref{thmintro.unirep.Isom(UQ)}, it remains to show that every continuous unitary representation of $\Isom(\U_\Delta)$ is dissociated. Using the Kat{\v e}tov construction, we show in Section \ref{sec.Urysohn} that $\Isom(\U_\Delta)$ can be approximated in a certain sense by an increasing sequence of oligomorphic groups, which we obtain as isometry groups of Urysohn spaces with \emph{finite} distance sets. Combining this observation with the fact that dissociation holds for every continuous unitary representation of oligomorphic groups (Proposition 3.2 of \cite{JT}), we prove in Theorem \ref{thm.dissociation.Isom(U_Delta)} that dissociation also holds for every unitary representation of $\Isom(\U_\Delta)$. The notion of approximation and the fact that dissociation passes to the limit is made rigorous in Section \ref{sec.approximating.sequence}.

We believe that the methods we develop to obtain dissociation - and therefore classification of unitary representations - through an approximation argument have a considerable potential for further applications. One possible class of structures for which these methods could apply is that of countable relational structures with a \emph{stationary independence relation}. Such structures can be approximated by substructures (see \cite{Muller}) in a way that generalizes the approximation argument that we use for $\Isom(\U_\Delta)$ in the present article.

\paragraph{Ergodic theoretic consequences.}

Besides being relevant for classifying unitary representations as explained in Theorem \ref{thmintro.classification.dissociated}, the notion of dissociation turns out to be appropriate for various other questions in ergodic theory. Let $(X,\mu)$ be a standard probability space. A p.m.p.\ action of a topological group $G$ on $(X,\mu)$ is an action of $G$ on $X$ by Borel automorphisms, such that the probability measure $\mu$ is $G$-invariant. A p.m.p.\ action $G\curvearrowright (X,\mu)$ leads to a unitary representation via the Koopman representation $\kappa \colon G\mapsto\UU(\L^2_0(X,\mu))$ defined by $\kappa(g)\colon f\mapsto f(g\inv x)$. Here $\L^2_0(X,\mu)=\L^2(X,\mu)\ominus\C$. In the context of closed permutation groups, dissociation of the Koopman representation is equivalent to conditional independence of some sub-$\sigma$-algebras of the Borel $\sigma$-algebra on $X$. We explain this correspondence here. Let $G\leq\Sym(\Omega)$ be a closed permutation group and $G\curvearrowright (X,\mu)$ a p.m.p.\ action. For every finite subset $A\subseteq\Omega$, we denote by $\FF_A$ the $\sigma$-algebra of measurable subsets $Y\subseteq X$ that are $G_A$-invariant, in the sense that for every $g\in G, \mu(gY\triangle Y)=0$. Say that the p.m.p.\ action $G\curvearrowright(X,\mu)$ is dissociated\footnote[1]{Dissociation was defined for ergodic p.m.p.\ actions in \cite{jahel2023stabilizers} as follows: for all $A,B\subseteq\Omega$ finite \emph{disjoint}, $\FF_A$ and $\FF_B$ are independent. We do believe that the present definition, which is stronger, is the correct one and we will amend our previous work to reflect this fact.} if for all finite subsets $A,B\subseteq\Omega$, the $\sigma$-algebras $\FF_A$ and $\FF_B$ are independent conditionally on $\FF_{A\cap B}$. Recall that given three sub-$\sigma$-algebras $\FF_1,\FF_2,\FF_3$, we say that $\FF_1$ and $\FF_3$ are independent conditionally on $\FF_2$ and write $\FF_1\indep_{\FF_2}\FF_3$ if for every $Y\in\FF_3$, we have $\P(Y\mid \sigma(\FF_1,\FF_2))=\P(Y\mid\FF_2)$. 
Conditional independence is related to conditional orthogonality (see \cite[Thm.~8.13]{KallenbergFMP}) in a way that leads to the following.

\begin{fact}
    The p.m.p.\ action $G\curvearrowright(X,\mu)$ is dissociated if and only if its  Koopman representation $\kappa \colon G\to\UU(\L^2_0(X,\mu))$ is dissociated. 
\end{fact}

In the work \cite{jahel2023stabilizers} of the last two authors, we show how dissociation for p.m.p.\ actions leads to stabilizer rigidity results \emph{à la }Stuck-Zimmer \cite{StuckZimmer}. Using Theorem 1.4 in \cite{jahel2023stabilizers}, the model theoretic properties satisfied by $\Isom(\U_\Delta)$ (see Lemma \ref{lem.Isom(UDelta).model.theory}) and our result that every unitary representation of $\Isom(\U_\Delta)$ is dissociated (see Theorem \ref{thm.dissociation.Isom(U_Delta)}), we obtain the following result. 

\begin{theorem}\label{thmintro.stabilizer.rigidit.Isom(UQ)}
   Let $\Delta\subseteq\R_+$ be a distance set. Then every ergodic p.m.p.\ action of $\Isom(\U_\Delta)$ is either essentially free (a conull set of points have a trivial stabilizer) or essentially transitive (one orbit has full measure). 
\end{theorem}

Let us go through one straightforward consequence of Theorem \ref{thmintro.stabilizer.rigidit.Isom(UQ)} concerning invariant random subgroups of $\Isom(\U_\Delta)$. Let $G$ be a Polish group. Let $\Sub(G)$ be the space of all closed subgroups of $G$. There is a natural $\sigma$-algebra on $\Sub(G)$ called the Effros $\sigma$-algebra which turns $\Sub(G)$ into a standard Borel space. An \defin{invariant random subgroup} (IRS for short) of $G$ is a Borel probability measure on $\Sub(G)$ which is invariant by conjugation. An IRS $\nu$ is concentrated on a conjugacy class if there exists an orbit $O$ of the $G$-action by conjugation on $\Sub(G)$ such that $\nu(O)=1$. For more about IRSs of Polish groups, we refer to \cite{jahel2023stabilizers}. A direct consequence of Theorem \ref{thmintro.stabilizer.rigidit.Isom(UQ)} is the following. 

\begin{corollary}\label{corintro.IRS}
    Let $\Delta\subseteq\R_+$ be a distance set. Then every invariant random subgroup of $\Isom(\U_\Delta)$ is concentrated on an orbit.
\end{corollary}

We finish by another ergodic theoretic application of the notion of dissociation, that of de Finetti's theorem. Let $G\leq\Sym(\Omega)$ be a closed permutation group. For every standard Borel space $Z$, the group $G$ acts on $Z^\Omega$ by shifting coordinates: for all $g\in G$ and $(z_\omega)_{\omega\in\Omega}\in Z^\Omega$, \[g\cdot (z_{\omega})_{\omega\in\Omega}=(z_{g\inv(\omega)})_{\omega\in\Omega}.\] 
Classifying ergodic probability measures on $Z^\Omega$ that are invariant under this action is a main problem in exchangeability theory which dates back to de Finetti. When $G$ acts transitively on $\Omega$ and $\mu$ is an ergodic probability measure on $Z^\Omega$ such that the p.m.p.\ action $G\curvearrowright (Z^\Omega,\mu)$ is dissociated, then $\mu$ is clearly a product measure of the form $\lambda^{\Omega}$ for some Borel probability measure $\lambda$ on $Z$. Since we prove that every p.m.p.\ action (and in fact every unitary representation) of $G=\Isom(\U_\Delta)$ is dissociated, we obtain the following theorem. 

\begin{theorem}\label{thmintro.deFinetti}
    Let $\Delta\subseteq\R_+$ be a distance set. Let $Z$ be a standard Borel space. Then the only ergodic probability measures on $Z^{\U_\Delta}$ that are invariant under the shift action $\Isom(\U_\Delta)\curvearrowright Z^{\U_\Delta}$ are product measures of the form $\lambda^{\U_\Delta}$, where $\lambda$ is a Borel probability measure on $Z$.
\end{theorem}

\begin{remark}
    A different proof of this result can be obtained via the methods from \cite{BJM}. For this, one has to check that $\mathbb{U}_\Delta$ is $1$-overlap closed. This is done in Section 6 of \cite{AKL}.
\end{remark}

\paragraph{Acknowledgements} We thank Julien Melleray for a discussion around topological simplicity of the isometry group of the Urysohn space $\U$ and François le Maître for many remarks on the paper. Part of this work was conducted while M.J.\ was visiting C.J. in Dresden. M.J. would like to thank the team of the Institute of Algebra for its warm welcome. C.J. was funded by the Deutsche Forschungsgemeinschaft (DFG, German Research Foundation) – project number 467967530.

\section{Background on representation theory}

In this article, Hilbert spaces are always complex. For such a space $\HH$, we denote by $\BB(\HH)$ the space of linear operators $T \colon \HH\to\HH$ which are bounded in the sense that there exists a constant $C>0$ such that for every $\xi\in\HH$, $\lVert T\xi\rVert \leq C\lVert\xi\rVert$. The strong operator topology on $\BB(\HH)$ is the topology induced by the seminorms $T\mapsto \lVert T\xi\rVert$ for $\xi\in\HH$. Examples of bounded operators are orthogonal projections. 
For a closed subspace $\KK\subseteq\HH$, we denote by $p_\KK\in\BB(\HH)$ the orthogonal projection onto $\KK$. Recall that $p_\KK\xi$ is the unique element of $\KK$ minimizing $\lVert\xi-\cdot\rVert$ and that it satisfies $p_\KK+p_{\KK^\bot}=\mathrm{id}_\HH$. The following result is the Hilbertian version of the classical reverse martingale convergence theorem. 

\begin{lemma}[Hilbertian reverse martingale convergence theorem]\label{lem.martingale}
Let $(\HH_n)_{n\geq 0}$ be a decreasing sequence of closed subspaces of a Hilbert space $\HH$ and let $\HH_\infty\coloneqq \bigcap_{n\geq 0}\HH_n$. Then $p_{\HH_n}\to p_{\HH_\infty}$ in the strong operator topology.  
\end{lemma}

If $\HH_1,\HH_2,\HH_3$ are three closed subspaces of $\HH$ satisfying $\HH_2\subseteq\HH_1\cap\HH_3$, we say that $\HH_1$ and $\HH_3$ are \defin{orthogonal conditionally on} $\HH_2$ and write $\HH_1\bot_{\HH_2}\HH_3$ if the subspaces $\HH_1\cap(\HH_2)^\bot$ and $\HH_3\cap (\HH_2)^\bot$ are orthogonal. The following lemma is a useful (and straightforward) characterization of conditional orthogonality using orthogonal projections.

\begin{lemma}\label{lem.conditional.orthogonality}$\HH_1\bot_{\HH_2}\HH_3$ if and only if $p_{\HH_1}p_{\HH_3}=p_{\HH_2}$. If this holds,  $\HH_2 = \HH_1\cap\HH_3$.
\end{lemma}

Let $\UU(\HH)$ be the unitary group of $\HH$. Equipped with the strong operator topology, this is a Polish group. Let $G$ be a topological group. A \defin{unitary representation} of $G$ is a homomorphism from $G$ to the unitary group $\UU(\HH)$. \emph{In this article, unitary representations will always be continuous homomorphisms. Since we will be dealing with Polish groups, we may always assume that Hilbert spaces are separable}. 

Let $\pi \colon G \to \UU(\HH)$ be a unitary representation of a topological group $G$. Given a subgroup $K\leq G$, we will denote by $\HH^K$ the closed subspace of $\pi(K)$-invariant vectors and by $p_K\in\BB(\HH)$ the orthogonal projection onto $\HH^K$. For $\xi\in\HH$,  $\conv G \xi$ denote the closed convex hull of $\pi(G)\xi$. The \defin{$G$-cyclic hull} of $\xi$ is the closure of the linear span of the orbit of $\xi$ under $G$. Similarly, if $\KK$ is a subspace of $\HH$, the $G$-cyclic hull of $\KK$ is the closure of the linear span of $G\cdot \KK = \{\pi(g)\xi, \ \xi \in \KK\}$. The vector $\xi$ (resp.\ the subspace $\KK$) is said to be \defin{$G$-cyclic in $\HH$} if the $G$-cyclic hull of $\xi$ (resp.\ of $\KK$) is $\HH$. 

The following useful theorem is due to Alaoglu-Birkhoff \cite{AlaogluBirkhoff}. To make the treatment comprehensive, we include a proof.

\begin{theorem}[Alaoglu-Birkhoff]\label{Th.BirkAl}
    Let $\HH$ be a Hilbert space and $G$ be any subgroup of $\UU(\HH)$. For every $\xi \in \HH$, $p_{G}\xi$ is the unique vector of minimal norm in $\conv G \xi$. In particular, $p_G\xi$ lies in the $G$-cyclic hull of $\xi$.
\end{theorem}

\begin{proof}
    Fix $\xi\in\HH$. Let $\xi_0\in \conv G \xi$ be the unique vector of minimal norm. Since $\conv G \xi$ is $G$-invariant, then $\xi_0$ belongs to $\HH^G$. We first show that $p_G\xi=\xi_0$. Notice that $0\in\conv G {\xi-\xi_0}=\conv G \xi - \xi_0$. So we may assume that $\xi_0=0$. For every $\varepsilon>0$, there exists $T_1,\dots,T_n\in G$ and $\lambda_1,\dots,\lambda_n\in\R_+$ such that $\lambda_1+\dots+\lambda_n=1$ and 
    \[\left\lVert\sum_{i=1}^n\lambda_i T_i\xi\right\rVert \leq\varepsilon.\]
    Therefore, for every $\eta\in\HH^G$, we have 
    \begin{align*}
        \lvert\langle \xi\mid\eta\rangle\rvert &= \left\lvert \sum_{i=1}^n \lambda_i\langle \xi \mid T_i\inv\eta\rangle\right\rvert \\ &=\left\lvert \langle\sum_{i=1}^n \lambda_i T_i\xi \mid \eta\rangle\right\rvert \\
        &\leq \varepsilon\lVert\eta\rVert.
    \end{align*}
 Thus, $\xi$ is orthogonal to $\HH^G$ and this proves that $p_G\xi=0$. 

    Finally, let us prove that $\xi_0$ is the unique $G$-invariant vector of $\conv G \xi$.
    Fix $\xi_1\in \conv G \xi\cap\HH^G$. For $i\in\{1,2\}$, define $\eta_i=\xi-\xi_i$. Define $\eta=\eta_0-\eta_1=\xi_1-\xi_0$ and observe that $\eta\in\HH^G$. Fix $t\in\R$. For $T_1,\dots,T_n\in G$ and $\lambda_1,\dots,\lambda_n\in\R_+$ satisfying $\lambda_1+\dots+\lambda_n=1$, we have
    \[\left\lVert\sum_{i=1}^n\lambda_iT_i(\eta+t\eta_0)\right\rVert\leq \lVert\eta+t\eta_0\rVert.\]
    Moreover, the left-hand side of the inequality can be made arbitrarily close to $\lVert \eta\rVert$ since $0\in\conv G {\eta_0}=\conv G \xi - \xi_0$. Thus, for all $t\in\R$, we have $\lVert \eta\rVert\leq \lVert \eta+t\eta_0\rVert$. Similarly, for all $t\in\R$, we have $\lVert \eta\rVert\leq \lVert \eta+t\eta_1\rVert$. We therefore get that $\langle\eta\mid\eta_0\rangle=\langle\eta\mid\eta_1\rangle=0$ and thus $\lVert\eta\rVert^2=\langle\eta\mid\eta_0-\eta_1\rangle=0$. This shows that $\xi_0$ is the unique element of $\conv G \xi\cap\HH^G$.  
\end{proof}

A subrepresentation of  a unitary representation $G\to\UU(\HH)$ is a closed, $G$-invariant vector subspace of $\HH$. A unitary representation is \defin{irreducible} if its only subrepresentations are $\{0\}$ and $\HH$. Note that a closed subspace $\KK\subseteq \HH$ is $G$-invariant if and only if $p_\KK$ and $\pi(g)$ commute for every $g \in G$.

\begin{lemma}
\label{lem.projections.commute} Let $G$ be a topological group and $\pi\colon G\to \UU(\HH)$ a unitary representation. Let $K\leq G$ be a subgroup and $\KK\subseteq\HH$ be a subrepresentation. Then the following holds
\begin{enumerate}
    \item\label{item.commutation} $p_K$ and $p_\KK$ commutes.
    \item If $q_K\in\BB(\KK)$ denotes the orthogonal projection onto $\KK^K$, then $(p_K)_{\mid\KK}=q_K$.
\end{enumerate}  
\end{lemma}

\begin{proof}
    \begin{enumerate}
        \item Let $\xi \in \HH$. Since $p_\KK$ is continuous and commutes with $\pi(g)$ for every $g \in G$, it must satisfy $p_\KK(\conv K \xi ) \subseteq\conv K {p_\KK\xi}$. As $p_\KK p_K\xi$ is clearly $K$-invariant, we conclude by Theorem \ref{Th.BirkAl}.

        \item Fix $\xi\in\KK$. By \ref{item.commutation}, $p_K\xi=p_Kp_\KK\xi=p_{\KK}p_K\xi\in\KK$. Therefore, $p_K\xi\in\KK^K$. We deduce that $p_K\xi$ is the unique element of $\KK^K$ minimizing $\lVert\xi-\cdot\rVert$, that is $p_K\xi=q_K\xi$.\qedhere
        
        \end{enumerate}
\end{proof}

One of the main techniques in representation theory is that of induction. Let $G$ be a topological group and $H\leq G$ an open subgroup. Let $\sigma \colon H\to\UU(\KK)$ be a unitary representation. Let $\EE$ be the space of maps $f\colon G\to\KK$ such that for every $g\in G$, $h\in H$, $f(gh)=\sigma(h\inv)f(g)$. Notice that the map $g\mapsto\lVert f(g)\rVert$ is constant on each left $H$-coset. Denote by $\lVert f(q)\rVert$ its value on the coset $q\in G/H$. 
Let $\HH$ be the Hilbert space of all $f\in \EE$ such that 
\[\sum_{q\in G/H}\lVert f(q)\rVert^2 <+\infty.\]
The induced representation $\pi\coloneqq\mathrm{Ind}_H^G(\sigma)$ is the representation of $G$ on $\HH$ defined by 
\[\pi(g)f \colon x\mapsto f(g\inv x)\quad\text{ for all }g\in G, f\in\HH.\]
Since $H$ is open, $\Ind_H^G(\sigma) \colon G\to\UU(\HH)$ is indeed continuous. For a map $f\in\HH$, we will denote by $\supp(f)$ the set of elements $g\in G$ such that $f(g)\neq 0$. Here are some classical general facts about induced representations. We refer for instance to \cite[Sec.~1.F]{BdlH} for more details.

\begin{lemma}\label{lem.induced.classical}
    Let $G$ be a topological group and $H\leq G$ an open subgroup. Let $\sigma \colon H\to\UU(\KK)$ be a unitary representation and $\pi = \Ind_H^G(\sigma)$. Write $\HH$ for the underlying Hilbert space of the representation $\pi$. Let  $\HH_0 \coloneq \{f\in \HH \colon \supp(f) \subseteq H\}$. 

    \begin{enumerate}
        \item $\HH_0$ is a closed subspace of $\HH$ that is stable under action of $H$. 

        \item The restriction of $\pi$ to $H$ and $\HH_0$ is canonically isomorphic to $\sigma$.

        \item $\HH_0$ is $G$-cyclic in $\HH$.
    \end{enumerate}
\end{lemma}

A topological group is \defin{non-archimedean} if it admits a basis of neighborhoods consisting of open subgroups. The class of topological groups under consideration in this article is the class of non-archimedean Polish groups, which coincides with that of closed permutation groups \cite[Thm.~1.5.1]{BeckerKechris}. For such groups, the following classical lemma (a proof of which can be found for instance in \cite[Lem.~3.1]{Tsankov}) turns out to be very useful.

\begin{lemma}\label{lem.continuity}
    Let $G$ be a topological group and let $(V_i)_{i\geq 0}$ be a countable basis of neighborhood of the identity consisting of open subgroups. Then for every unitary representation $G\to\UU(\HH)$, the space $\bigcup_{i\geq 0}\HH^{V_i}$ is dense in $\HH$. 
\end{lemma}

Given a closed permutation group $G\leq\Sym(\Omega)$, a continuous unitary representation $\pi \colon G\to\UU(\HH)$ and a finite subset $A\subseteq\Omega$, we denote (for conciseness) by $\HH_A$ the subspace $\HH^{G_A}$, that is, the subspace of $G_A$-invariant vectors in $\HH$. Similarly, we denote by $p_A\in\BB(\HH)$ the orthogonal projection onto $\HH_A$. The subspaces $\HH_A$ play an important role in the definition of dissociation that is the main topic of the next section.

\section{Dissociated unitary representations}\label{sec.dissociation.unirep}

\subsection{Dissociation and induced representations}

We introduce here a structural property for unitary representations of closed permutation groups that we call dissociation. As the present paper focuses on $\Isom(\Q\U)$ which has no algebraicity and  weakly eliminates imaginaries (see Lemma \ref{lem.Isom(UDelta).model.theory}), the following definition is tailored for groups satisfying these two assumptions (see Remark \ref{rem.dissociation.NoAlg.WEI}). Without these assumptions, dissociation can be defined by mimicking Proposition 3.2 of \cite{JT}; this will be the topic of a future work. 

\begin{definition}
    Let $G\leq\Sym(\Omega)$ be a closed permutation group. A unitary representation $\pi \colon G\to\UU(\HH)$ is \defin{dissociated} if for all finite subsets $A,B\subseteq\Omega$, the subspaces $\HH_A$ and $\HH_B$ are orthogonal conditionally on $\HH_{A\cap B}$. 
\end{definition}

\begin{remark}\label{rem.dissociation.NoAlg.WEI}
    Let $G\leq \Sym(\Omega)$ be a closed permutation group acting without fixed points on $\Omega$. Assume that for every unitary representation $\pi \colon G\to\UU(\HH)$ and all finite subsets $A,B\subseteq\Omega$, we have $\HH_A\bot_{\HH_{A\cap B}}\HH_B$. Fix any two finite subsets $A,B\subseteq\Omega$. By dissociation of the quasi-regular representation $G\to\UU(\ell^2(G/{\langle G_A,G_B\rangle}))$, we get that $\HH^{\langle G_A,G_B\rangle}=\HH_A\cap\HH_B=\HH_{A\cap B}$. Thus, $\langle G_A,G_B\rangle = G_{A\cap B}$. Therefore, $G$ has no algebraicity and weakly eliminates imaginaries by Lemma \ref{lem.NoAlg.WEI}.
\end{remark}

\begin{remark}
    A closed permutation group $G\subseteq \Sym(\Omega)$ is \defin{oligormorphic} if for every $n\in \N$, the diagonal action $G \curvearrowright \Omega^n$ has finitely many orbits. When $G$ is the automorphism group of a first order structure with domain $\Omega$ in a countable signature, this is equivalent to the structure being $\aleph_0$-categorical by the celebrated Ryll-Nardzewski Theorem \cite[Thm.~7.3.1]{Hodges_1993}. Examples of such structures include the countable set, the rationals with the usual order, the Rado graph, the countable vector space over a finite field, the countable atomless boolean algebra and $\Delta$-Urysohn spaces where $\Delta$ is a finite distance set. 
    
    Oligomorphic groups with no algebraicity that admit weak elimination of the imaginaries is an important class of examples of groups for which \textit{all} unitary representations are dissociated (see Proposition 3.2 in \cite{JT}). Notice that dissociation is obtained in \cite{JT} as a corollary of the classification due to Tsankov \cite{Tsankov} of unitary representations of oligomorphic groups. We go here in the other direction, by proving that the abstract notion of dissociation implies a classification of the unitary representations. This method applies to new examples such as $\Isom(\U_\Delta)$ for \textit{every} countable distance set $\Delta$. 
\end{remark}

We prove below that dissociation passes to subrepresentations.
\begin{lemma}\label{lem.dissociation.subrepresentation}
    Let $G\leq\Sym(\Omega)$ be a closed permutation group. Let $\pi \colon G\to\UU(\HH)$ be a unitary representation. If $\pi$ is dissociated, then every subrepresentation of $\pi$ is dissociated. 
\end{lemma}

\begin{proof}
    Let $\KK\subseteq\HH$ be a subrepresentation of $\pi$. 
    Let $A,B\subseteq\Omega$ be two finite subsets. Write $p_A,p_B,p_{A\cap B}\in\BB(\HH)$ for the orthogonal projections onto $\HH_A$, $\HH_B$ and $\HH_{A\cap B}$ and write $q_A,q_B,q_{A\cap B}\in\BB(\KK)$ for the orthogonal projections onto $\KK_A$, $\KK_B$ and $\KK_{A\cap B}$. For $\xi\in\KK$, we have that
    \[q_Aq_B\xi = p_Ap_B\xi = p_{A\cap B}\xi = q_{A\cap B}\xi \]
    where the first and last equality uses Lemma \ref{lem.projections.commute} and the middle one uses the assumption that $\pi$ is dissociated. Thus, the subrepresentation $\KK$ is dissociated. 
\end{proof}

The following theorem is a essential step towards classifying dissociated unitary representations. A similar method has been used by Ol'shanskii in \cite[Lem~2.2]{Olshanskii} for $G=\Sym(\Omega)$.

\begin{theorem}\label{thm.dissociation.implies.induced}
    Let $G\leq\Sym(\Omega)$ be a closed subgroup. Let $\pi \colon G\to\UU(\HH)$ be a non-zero unitary representation. If $\pi$ is dissociated, then $\pi$ contains a non-zero subrepresentation which is of the form $\Ind_{G_{\{A\}}}^G(\sigma)$ for some finite non-empty subset $A\subseteq\Omega$ and some unitary representation $\sigma$ of $G_{\{A\}}$ that is trivial on $G_A$. 
\end{theorem} 

\begin{proof}
    By Lemma \ref{lem.continuity}, fix $A\subseteq \Omega$ finite with minimal cardinality for the property $\HH_A \neq \{0\}$. Then $\HH_A$ is a closed subspace of $\HH$ stable under the action of $\Nm G A$. Thus restricting $\pi$ gives rise to a representation $\sigma$ of $\Nm G A$ on $\HH_A$, which is trivial on $G_A$. Let us show that $\mathrm{Ind}_{\Nm G A}^G(\sigma)$ is a subrepresentation of $\pi$.

    \begin{claim*}
        For all $g,h\in G$ such that $h\inv g\notin \Nm G A$, we have $\pi(g)\HH_A \perp \pi(h)\HH_A$.
    \end{claim*} 

\begin{cproof}
    Notice that $\pi(g)\HH_A = \HH_{gA}$ and similarly for $h$. Since $\pi$ is dissociated, 
    $$\HH_{gA}\perp_{\HH_{gA \cap hA}} \HH_hA.$$
    Assuming $h^{-1}g\notin \Nm G A$, then $\ gA \neq  hA$ and $|gA \cap hA|<|A|$. In particular, we get $\HH_{gA\cap hA} =\{0\}$ by minimality of $|A|$ and the claim is proved.
\end{cproof}

    Next, denote by $\KK$ the underlying Hilbert space of $\mathrm{Ind}_{\Nm G A}^G(\sigma)$ and let $(g_i)_{i\in I}$ be a system of representatives of left $\Nm G A$-cosets in $G$. The previous claim ensures that the following map is well defined and isometric:
    $$\KK \longrightarrow \HH , \qquad f \longmapsto \sum_{i\in I} \pi(g_i)f(g_i).$$
    It is easily seen that this map does not depend on the choice of $(g_i)_{i\in I}$ and, hence, that it is $G$-equivariant. Therefore, $\Ind_{G_{\{A\}}}^G(\sigma)$ is a subrepresentation of $\pi$. 
\end{proof}

The second step towards classifying dissociated unitary representations is a fine analysis of induced representations of the form $\Ind_{G_{\{A\}}}^G(\sigma)$.

\begin{remark} \label{rem.quotient}
Note that if $H\trianglelefteq K$ are topological groups, then representations of $K$ that are trivial on $H$ and representations of $K/H$ are really the same thing. We will make implicit use of this observation when stating results and manipulating representations. 

In our context, if $G\leqslant \Sym(\Omega)$ is a closed permutation group and $A \subseteq \Omega$ is finite, then $\Nm G A / G_A$ is a finite group, which naturally identifies as a subgroup of $\Sym(A)$. In particular, every irreducible representation of $\Nm G A$ which is trivial on $G_A$ is finite dimensional.
\end{remark}

\begin{lemma}\label{lem.induction.analyzed}
    Let $G\leq\Sym(\Omega)$ be a closed permutation group with no algebraicity. Let $A\subseteq\Omega$ be a finite subset and $\sigma$ a unitary representation of $G_{\{A\}}/G_A$. Consider $\pi=\Ind_{G_{\{A\}}}^G(\sigma)$ and denote by $\HH$ its underlying Hilbert space. Then the following holds. 
    \begin{enumerate}
        \item\label{item.HB} For every  $B\subseteq\Omega$ finite, we have $\HH_B=\{f\in\HH\colon \supp(f)\subseteq\{g\in G\colon gA\subseteq B\}\}$.  In particular, $\HH_A = \{f\in \HH\colon \supp(f) \subseteq \Nm G A\}$ and the restriction of $\pi$ to $\HH_A$ and $\Nm G A$ is isomorphic to $\sigma$.
        
        \item\label{item.projection.is.multiplication} For every $B\subseteq\Omega$ finite, $p_B\in\BB(\HH)$ is the multiplication by $\mathds{1}_{\{g\in G\colon gA\subseteq B\}}$.
        
        \item\label{item.induced.irreducible} $\pi$ is irreducible if and only if $\sigma$ is.
    \end{enumerate}
\end{lemma}

\begin{proof}
\begin{enumerate}
    \item Let $f\in\HH_B$. Notice that the map $g\mapsto \lVert f(g)\rVert$ is constant on each double coset of the form $G_BgG_{\{A\}}$. Since $\lVert f\rVert^2 =\sum_{q\in G/G_{\{A\}}}\lVert f(q)\rVert^2$ is finite, this shows that for every $g\in\supp(f)$, the double coset $G_{B}gG_{\{A\}}$ is a disjoint union of finitely many left cosets of $G_{\{A\}}$. Since $A$ is finite, this is equivalent to saying that for every $a\in A$, $G_Bg(a)$ is finite. Using that $G$ has no algebraicity, we get that $gA\subseteq B$ and thus $\supp(f)\subseteq\{g\in G\colon gA\subseteq B\}$. 

    Conversely, let $f\in\HH$ be such that $\supp(f)\subseteq\{g\in G\colon gA\subseteq B\}$. Fix $h\in G_B$. We need to prove that $\pi(h)f=f$, that is, for every $g\in G$, $f(h\inv g)=f(g)$. There are two cases to check.
    
    \begin{itemize}
    \item If $gA\not\subseteq B$, then $h\inv gA\not\subseteq B$ and thus $f(h\inv g) = 0=f(g)$. 

    \item If $gA\subseteq B$, then $h\inv g\in G_Bg$. But $G_Bg=gG_{g\inv B}\subseteq gG_A$. So there exists $k\in G_A$ such that $h\inv g=gk$. Thus
    \[f(h\inv g)=\sigma(k\inv)(f(g))=f(g)\]
    since $\sigma$ is trivial on $G_A$. 
    \end{itemize}

    Applying the above to $B=A$, we indeed obtain $\HH_A = \{f\in \HH\colon \supp(f) \subseteq \Nm G A\}$. The last claim then follows from Lemma \ref{lem.induced.classical}.

    \item Fix $B\subseteq\Omega$ finite and $f\in\HH$. Let $M$ be the sum of $\lVert f(q)\rVert^2$ for every $q\in G/G_{\{A\}}$ satisfying $q\cap\{g\in G\colon gA\subseteq B\} = \emptyset$. Using the description of $\HH_B$ obtained above, it is clear that for every $f'\in\HH_B$, we have 
    $\lVert f-f'\rVert^2\geq M$ with equality if and only if $f'=f\cdot\mathds{1}_{\{g\in G\colon gA\subseteq B\}}$, which proves the result. 
    
    \item Assume that $\sigma$ is irreducible. Note that a representation is irreducible if and only if every non-zero vector is cyclic. We first proof the following:
    \begin{claim*}
        Every non-zero vector in $\HH_A$ is $G$-cyclic in $\HH$.
    \end{claim*}
    \begin{cproof}
        Using Item \ref{item.HB} and Lemma \ref{lem.induced.classical}, we get that $\Nm G A \curvearrowright \HH_A$ is irreducible. Thus every non-zero vector in $\HH_A$ is $\Nm G A$-cyclic in $\HH_A$. But $\HH_A$ is $G$-cyclic in $\HH$ by Item \ref{item.HB} and Lemma \ref{lem.induced.classical}.
    \end{cproof}
    
    Fix $f \in \HH\backslash\{0\}$. Up to translating $f$ using $\pi$, we can assume that $f(e_G)\neq 0$. It follows from Item \ref{item.projection.is.multiplication} that  $p_A$ is the multiplication by $\mathds{1}_{\Nm G A}$. In particular, $p_Af\neq 0$. By the claim, $p_Af$ is a $G$-cyclic vector in $\HH$ that lies in the cyclic hull of $f$ by the Alaoglu-Birkhoff theorem. Necessarily, $f$ is also $G$-cyclic in $\HH$ and $\pi$ is irreducible.
    
    The converse is straightforward since induction preserves subrepresentations, see for instance \cite[Cor.~E.2.3]{BdlHV}.\qedhere
    \end{enumerate}
\end{proof}

\begin{corollary}\label{cor.induced.are.dissociated}
    Let $G\leq\Sym(\Omega)$ be a closed permutation group with no algebraicity. Let $A\subseteq\Omega$ be a finite subset and $\sigma$ a unitary representation of $G_{\{A\}}/G_A$. Then $\Ind_{G_{\{A\}}}^G(\sigma)$ is dissociated. 
\end{corollary}

\begin{proof}
   Using Item \ref{item.projection.is.multiplication} of Lemma \ref{lem.induction.analyzed}, we readily get that for all finite subsets $B,C\subseteq\Omega$, $p_Bp_C=p_{B\cap C}.$ Therefore, $\HH_B\bot_{\HH_{B\cap C}}\HH_C$ and thus $\Ind_{G_{\{A\}}}^G(\sigma)$ is dissociated.
\end{proof}

We therefore obtain a complete classification of the dissociated representations of a permutation group with no algebraicity, which is very similar to Tsankov's classification of unitary representations for oligomorphic groups \cite{Tsankov}. Given a subgroup $K$ of $G$ and an element $g\in G$, we will write $K^g = gKg^{-1}$ and denote by $\sigma^g$ the unitary representation of $K^g$ given by $\sigma^g(u) = \sigma(g^{-1}ug)$ for every $u\in H^g$.

\begin{theorem}\label{thm.classification.dissociated}
    Let $G\leq\Sym(\Omega)$ be a closed permutation group without algebraicity.
    \begin{enumerate}
        \item The dissociated irreducible unitary representations of $G$ are exactly the unitary representations isomorphic to one of the form $\Ind_{G_{\{A\}}}^G(\sigma)$ where $A$ ranges over the finite subsets of $\Omega$ and $\sigma$ over the irreducible representations of the finite group $G_{\{A\}}/G_A$.

        \item Two such irreducible representations $\Ind_{G_{\{A\}}}^G(\sigma)$ and $\Ind_{G_{\{ B\}}}^G(\tau)$ are isomorphic if and only if there exists $g\in G$ such that $gA = B$ and $\sigma^g \simeq \tau$.

        \item Every dissociated unitary representation of $G$ splits as direct a sum of irreducible subrepresentations.
    \end{enumerate}
\end{theorem}

\begin{proof}
    \;

    \begin{enumerate}
        \item Every irreducible unitary representation of $G$  has this form by Theorem \ref{thm.dissociation.implies.induced}. Moreover, all of these representations are irreducible by Item \ref{item.induced.irreducible} in Lemma \ref{lem.induction.analyzed}.

        \item Let $A, B$ be finite subsets of $\Omega$ and $\sigma, \tau$ be irreducible representations of $\Nm G A / G_A$ and $\Nm G B / G_B$ respectively. Assume that $\pi = \Ind_{G_{\{A\}}}^G(\sigma)$ and $\pi'=\Ind_{G_{\{ B\}}}^G(\tau)$ are isomorphic, with respective underlying Hilbert spaces $\HH$ and $\KK$. Then $\HH_B \simeq \KK_B$ which is non-zero by Item \ref{item.HB} in Lemma \ref{lem.induction.analyzed}, hence there exists $g \in G$ such that $gA\subset B$. By symmetry, $|A|=|B|$ and $gA = B$. Moreover, using Item \ref{item.HB} of Lemma \ref{lem.induction.analyzed} again and the fact that $\Nm G A^g = \Nm G {gA} = \Nm G B$, we get  
        $$\tau \simeq \left(\pi'\colon \Nm G B \curvearrowright \KK_B\right) \simeq \left(\pi\colon \Nm G B \curvearrowright \HH_B\right) = \left(\pi^g\colon \Nm G A \curvearrowright \HH_A\right) \simeq \sigma^g. $$

         \item Let $\pi \colon G\to\UU(\HH)$ be a non-zero dissociated unitary representation. By Theorem \ref{thm.dissociation.implies.induced}, $\pi$ contains a non-zero subrepresentation of the form $\Ind^G_{G_{\{A\}}}(\sigma)$ for some finite $A\subseteq\Omega$ and some unitary representation $\sigma$ of $G_{\{A\}}$ that factors through $G_{\{A\}}/G_A$. Since induction preserves subrepresentations and every unitary representation of a finite group contains an irreducible subrepresentation, we may assume that $\sigma$ is irreducible. Therefore $\Ind_{G_{\{A\}}}^G(\sigma)$ is irreducible by Item \ref{item.induced.irreducible} of Lemma \ref{lem.induction.analyzed}. Finally, one concludes using Zorn's Lemma and the fact dissociation passes to subrepresentations (Lemma \ref{lem.dissociation.subrepresentation}). \qedhere
    \end{enumerate}
\end{proof}

\subsection{Obtaining dissociation via approximating sequences}\label{sec.approximating.sequence}

Let $\Omega$ be a countably infinite set and $G\leq\Sym(\Omega)$ be a closed permutation group. Let $\Omega'\subseteq\Omega$ be an infinite subset and $H\leq \Sym(\Omega')$ a closed permutation group. An \defin{extension embedding} is an embedding of topological groups $\theta \colon H\hookrightarrow G$ such that for all $x\in X$, $h\in H$, $\theta(h)(x)=h(x)$. Fix an increasing sequence $\Omega_0\subseteq\Omega_1\subseteq\dots\subseteq\Omega$ of infinite subsets with $\Omega=\bigcup_{n\geq 0}\Omega_n$, a sequence of closed permutation groups $G_n\leq\Sym(\Omega_n)$ and a sequence of extension embeddings $\theta_n \colon G_n\hookrightarrow G_{n+1}$. Notice that for each $n\geq 0$, we can naturally define an extension embedding $\iota_n\colon G_n\hookrightarrow \Sym(\Omega)$ as follows: for all $g\in G_n$ and $x\in\Omega$, set
\begin{equation}\label{eq.extension.embedding}
    \iota_n(g)(x)\coloneqq\left\{\begin{array}{cl}
    g(x) & \text{if }x\in\Omega_n,\\
    (\theta_{m-1}\circ\dots\circ\theta_n(g))(x) & \text{if }x\in\Omega_m\text{ for some }m\geq n. 
\end{array}\right.
\end{equation}
It is clear that $\iota_n$ is well-defined and that this is an extension embedding. Notice moreover that  $\iota_n(G_n)$ is a subgroup of $\iota_{n+1}(G_{n+1})$ for every $n\geq 0$. 

\begin{definition}\label{def.approx.sequence}
Let $G\leq\Sym(\Omega)$ be a closed permutation group. An \defin{approximating sequence} for $G$ is the data of an increasing sequence $\Omega_0\subseteq\Omega_1\subseteq\dots\subseteq\Omega$ of infinite subsets with $\Omega=\bigcup_{n\geq 0}\Omega_n$, a sequence of closed permutation groups $G_n\leq\Sym(\Omega_n)$ and a sequence of extension embedding $\theta_n \colon G_n\hookrightarrow G_{n+1}$ such that $\bigcup_{n\geq 0}\iota_n(G_n)$ is a dense subgroup of $G$, where $\iota_n \colon G_n\to\Sym(\Omega)$ is the extension embedding defined in \eqref{eq.extension.embedding}.\end{definition}

In the sequel,  approximating sequences will be denoted by $G_0\hookrightarrow G_1\hookrightarrow\dots \hookrightarrow G$, the extension embeddings $G_n\hookrightarrow G_{n+1}$ by $\theta_n$ and the extension embeddings $G_n\hookrightarrow\Sym(\Omega)$ by $\iota_n$. 

\begin{lemma}\label{lem.approximating.seq.stabilizers}
    Let $G\leq\Sym(\Omega)$ be a closed permutation group with an approximating sequence $G_0\hookrightarrow G_1\hookrightarrow\dots\hookrightarrow G$. Let $A\subseteq\Omega$ be a finite subset and let $N\geq 0$ be such that $A\subseteq\Omega_N$. Then the sequence $(\iota_n(G_n)_A)_{n\geq N}$ of subgroups is increasing and  $\bigcup_{n\geq N}\iota_n(G_n)_A$ is a dense subgroup of $G_A$.  
\end{lemma}

\begin{proof}
    Notice that since $A\subseteq\Omega_N$, then for every $n\geq N$ we have $\iota_n(G_n)_A=\iota_n((G_n)_A)$. Thus, the sequence $(\iota_n(G_n)_A)_{n\geq N}$ is increasing. It is clear that $\bigcup_{n\geq N}\iota_n(G_n)_A$ is a subgroup of $G_A$. Let us prove that it is dense. Fix $g\in G_A$. By density of $\bigcup_{n\geq 0}\iota_n(G_n)$ in $G$, there a sequence $(g_k)_{k\geq 0}$ of elements in $\bigcup_{n\geq 0}\iota_n(G_n)$ such that $g_k\to g$. But the sequence $(\iota_n(G_n))_{n\geq 0}$ is increasing, so the $g_k$'s belong to $\bigcup_{n\geq N}\iota_n(G_n)$. Since $g\in G_A$ and $(g_k)_{k\geq 0}$ converges pointwise to $g$, then $g_k$ fixes $A$ pointwise eventually, which finishes the proof. 
\end{proof}

The main result of this subsection is the following, which establishes the fact that dissociation is closed under taking limits of approximating sequences.

\begin{theorem}\label{thm.dissociation.through.approximation} Let $G\leq\Sym(\Omega)$ be a closed permutation group with an approximating sequence $G_0\hookrightarrow G_1\hookrightarrow\dots\hookrightarrow G$. Assume that for every $n\geq 0$, every unitary representation of the permutation group $G_n\leq\Sym(\Omega_n)$ is dissociated. Then every unitary representation of $G\leq\Sym(\Omega)$ is dissociated. 
\end{theorem}

\begin{proof} Let $\pi \colon G\to\UU(\HH)$ be a unitary representation. Fix $A\subseteq\Omega$ finite and let $N\geq 0$ be such that $A\subseteq\Omega_N$. As usual, $\HH_A$ denotes the subspace of all vectors $\xi\in \HH$ such that $\pi(g)\xi=\xi$ for every $g\in G_A$. We denote by $p_A$ the orthogonal projection onto $\HH_A$. For $n\geq N$, we let $\HH_A^n$ be the set of all $\xi\in\HH$ such that $\pi(g)\xi=\xi$ for every $g\in \iota_n(G_n)_A$. We denote by $p_A^n$ the orthogonal projection onto $\HH_A^n$. 

\begin{claim*}
    $p_A^n\to p_A$ in the strong operator topology.
\end{claim*}

\begin{cproof}
    First,  $(\iota_n(G_n)_A)_{n\geq N}$ is increasing, so $(\HH_A^n)_{n\geq N}$ is decreasing. We claim that $\bigcap_{n\geq N}\HH_A^n=\HH_A$. We indeed have $\HH_A\subseteq\bigcap_{n\geq N}\HH_A^n$ since $\iota_n(G_n)_A$ is a subgroup of $G_A$ for every $n\geq N$. For the converse inclusion, fix $\xi\in\bigcap_{n\geq N}\HH_A^n$. Then for every $g\in\bigcup_{n\geq N}\iota_n(G_n)_A$, we have $\pi(g)\xi=\xi$. By Lemma \ref{lem.approximating.seq.stabilizers} and continuity of $\pi$, we get that $\pi(g)\xi=\xi$ for every $g\in G_A$ and thus $\xi\in\HH_A$. By the Hilbertian reverse martingale theorem (see Lemma \ref{lem.martingale}), we conclude that $p_A^n\to p_A$ in the strong operator topology. 
\end{cproof}
Fix $A,B\subseteq\Omega$ finite. We want to prove that $\HH_A\bot_{\HH_{A\cap B}}\HH_B$. For this, let us show that $p_Ap_B=p_{A\cap B}$. Fix $N\geq 0$ such that $A$ and $B$ belongs to $\Omega_N$. For every $n\geq N$, the unitary representation $\pi_n\colon G_n\overset{\iota_n}{\longrightarrow}\iota_n(G_n)\overset{\pi}{\longrightarrow} \UU(\HH)$ is continuous. Moreover, since for every $n\geq N$ we have $\iota_n(G_n)_A=\iota_n((G_n)_A)$ (and similarly for $B$ and $A\cap B$), we get that $\HH_A^n$ is the subspace of vectors that are invariant by $(G_n)_A$ (and similarly for $\HH_B^n$ and $\HH_{A\cap B}^n$). So by assumption, $\HH_A^n\bot_{\HH_{A\cap B}^n}\HH_B^n$. In other words, $p_A^np_{B}^n=p_{A\cap B}^n$. Since $||p^n_A||,||p^n_B||\leqslant 1$ for every $n\in \N$, their product also converges. Thus, $p_Ap_B=p_{A\cap B}$, which concludes the proof.
\end{proof}

\section{Rational Urysohn space and its fellows}\label{sec.Urysohn}

\subsection{Some properties of $\Isom(\U_\Delta)$}

In this section we recall some basic definitions related to variants of the rational Urysohn space. These objects are studied for instance in \cite{EGLMM}.

\begin{definition}
    A \defin{distance set} is 
    \begin{itemize}
        \item either a countable additive subsemigroup of the positive reals that contains $0$,
        \item or the intersection $\mathcal{S}\cap [0,r]$ of such a subsemigroup $\mathcal{S}$ and a bounded interval $[0,r]$ whith $r\in\mathcal{S}$.
    \end{itemize}
\end{definition}

A $\Delta$-metric space is a metric space whose metric takes its value in $\Delta$. Given a distance set $\Delta\subseteq\R_+$, the class of finite $\Delta$-metric spaces forms a Fra\"issé class (for a proof, see for instance \cite[Lem.~2.6]{EGLMM}) and we denote by $\U_\Delta$ its Fra\"issé limit. This is a $\Delta$-metric space which is called the Urysohn $\Delta$-metric space. It is the unique (up to isometry) countable $\Delta$-metric space satisfying the following two properties: 
\begin{itemize}
    \item (\emph{ultrahomogeneity}) Given any two finite subsets $A,B\subseteq \U_\Delta$ and any isometry $\varphi \colon A\to B$, there exists an isometry $\Tilde{\varphi} \colon \U_\Delta\to \U_\Delta$ which extends $\varphi$,
    \item (\emph{universality}) Every countable $\Delta$-metric space embeds isometrically into $\U_\Delta$. 
\end{itemize}

Let $\Isom(\U_\Delta)$ be the isometry group of $\U_\Delta$, which is equipped with the Polish topology of pointwise convergence ($\U_\Delta$ is here equipped with the discrete topology). 

\begin{example}
Here are some natural choices for $\Delta$.
    \begin{itemize}
        \item If $\Delta=\{0,r\}$, then $\U_\Delta$ is the countable discrete set. In this case $\Isom(\U_{\Delta})\simeq S_\infty$, the group of all permutations of a countably infinite set. 
        \item If $\Delta=\{0,r,2r\}$, then $\U_\Delta$ is the Rado graph. In fact, if we put an edge between $x,y\in\U_\Delta$ iff $d(x,y)=r$, then the structure obtained is isomorphic to the Rado graph. 
        \item If $\Delta=\Q_+$ , then $\U_\Delta$ is the rational Urysohn space $\Q\U$. 
        \item If $\Delta=\Z_+$, then $\U_\Delta$ is the integral Urysohn space $\Z\U$.
        \item If $\Delta=\Q_+\cap[0,1]$, then $\U_\Delta$ is the rational Urysohn sphere $\Q\U_1$. 
    \end{itemize}
\end{example}

We now discuss the geometrical properties of the groups $\Isom(\U_\Delta)$. These properties will not be used in the sequel, but they help situate these groups in the global picture of Polish groups. Recall that a subset $E$ of a topological group $G$ is \defin{Roelcke precompact} if for every neighborhood $U$ of the identify, there exists a finite subset $F\subseteq G$ such that $E \subseteq UFU$. The group $G$ is \defin{Roelcke precompact} if it is Roelcke precompact as a subset of itself. It is \defin{locally Roelcke precompact} if it admits a Roelcke precompact non-empty open subset. Moreover, a subset $E\subseteq G$ is \defin{coarsely bounded} if every left-invariant and continuous écart on $G$ assigns a finite diameter to $E$.

\begin{lemma}\label{lem.coarse.geometry}
    Let $\Delta\subseteq\R_+$ be a distance set. Then the following hold:
    \begin{enumerate}
        \item\label{item.loc.bounded} $\Isom(\U_\Delta)$ is locally bounded.
        \item\label{item.bounded} $\Isom(\U_\Delta)$ is coarsely bounded if and only if $\Delta$ is bounded.
        \item\label{item.loc.RP} $\Isom(\U_\Delta)$ is locally Roelcke precompact if and only if $\Delta$ is closed and discrete.
        \item\label{item.RP} $\Isom(\U_\Delta)$ is Roelcke precompact if and only if $\Delta$ is finite.
    \end{enumerate}

    \begin{proof}
        The proof of Items \ref{item.loc.bounded} and \ref{item.bounded} is contained in Theorem 6.31 and Examples 6.32 of \cite{Rosendal}. A topological group being Roelcke precompact if and only if it is both locally Roelcke precompact and coarsely bounded, Item \ref{item.RP} follows from Items \ref{item.bounded} and \ref{item.loc.RP}. 
        
        Thus, it only remains to prove Item \ref{item.loc.RP}. Let $G = \Isom(\U_\Delta)$. First, note that $\Delta$ is closed and discrete if and only if $\Delta\cap [0,M]$ is finite for every $M\geqslant 0$.
        
        Assume that there exists $M>0$ such that $\Delta\cap[0,M]$ is infinite. Let $U$ be an open neighborhood of the identity in $\Isom(\U_\Delta)$. Up to replacing $U$ with a smaller neighborhood, we can assume that $U$ is of the form $G_A$ for some finite set $A \subseteq\U_\Delta$. Let $b\in\U_\Delta$ be such that $d(b,A) \geqslant M/2$. By the Kat\v{e}tov construction and ultrahomogeneity, the distances between $b$ and elements of the $G_A$ orbit of $b$ take every value in the infinite set $\Delta\cap[0,M]$. Thus $G_{A\cup\{b\}}$ has infinitely many orbit on $G_A\cdot b$ and $G_A$ is not Roelcke precompact by \cite[Thm.~2.4]{Tsankov}.

    Conversely, assume $\Delta$ is closed and discrete. Fix $a\in\U_\Delta$ and let us show that $G_a$ is Roelcke precompact. For every $\delta\in\Delta$, define 
        \[ B(a,\delta)\coloneqq\{x\in\U_\Delta\colon d(a,x)\leq \delta\}.\]
        Let us prove that each action $G_a\curvearrowright B(a,\delta)$ is oligomorphic. Fix $\delta\in\Delta$ and recall that $\Delta\cap [0,\delta]$ is finite by assumption. For every $n\geq 1$, there are only finitely many isometric types of metric spaces of the form $(a,x_1,\dots,x_n)$ with $x_1,\dots,x_n\in B(a,\delta)$. Moreover, by ultrahomogeneity of $\U_\Delta$, any two such finite metric spaces which are isometric, are in the same $G_a$-orbit. Thus $G_a\curvearrowright B(a,\delta)$ is oligomorphic. Finally, $G_a$ is an inverse limit of oligomorphic groups and therefore is Roelcke precompact by \cite[Thm.~2.4]{Tsankov}. \qedhere
        
        \end{proof}

\end{lemma}
\begin{definition}
    Let $G\leq\Sym(\Omega)$ be  closed permutation group. We say that $G$ has \defin{weak elimination of imaginaries} if for every open subgroup $V\leq G$, there exists a finite subset $A\subseteq\Omega$ such that $G_A\leq V$ and $[V:G_A]<+\infty$. 
\end{definition}

Recall the following characterization obtained in \cite[Lem.~3.6]{jahel2023stabilizers}.

\begin{lemma}\label{lem.NoAlg.WEI}
    Let $G\leq\Sym(\Omega)$ be a closed permutation group. Then the following are equivalent. 
    \begin{enumerate}
        \item $G$ has no algebraicity and admits weak elimination of imaginaries.
        \item $G$ acts on $\Omega$ without fixed point and for all finite subsets $A,B\subseteq\Omega$, we have $\langle G_A,G_B\rangle=G_{A\cap B}$.
    \end{enumerate} 
\end{lemma}

Slutsky essentially proved that these conditions are satisfied for the isometry groups of Urysohn spaces. See \cite[Thm.~4.12, Thm.~4.16 and Cor.~4.17]{Slutsky}, from which we extract the following.

\begin{theorem}[Slutsky, \cite{Slutsky}]\label{thm.Slutsky}
    Let $\Delta\subseteq\R_+$ be a distance set and let $G=\Isom(\U_\Delta)$. Then for all $A,B\subseteq\U_\Delta$ finite, we have $\langle G_A,G_B\rangle = G_{A\cap B}$. 
\end{theorem}

\begin{remark}
    Slutsky's original result \cite[Cor.~4.17]{Slutsky} (see also \cite[Thm.~7.7]{EGLMM}) is stated by saying that $\langle G_A,G_B\rangle$ is dense in $G_{A\cap B}$. This is indeed what he proves. However, the subgroup $\langle G_A,G_B\rangle$ being open, it is also closed and the equality holds.
\end{remark}

Combining Lemma \ref{lem.NoAlg.WEI} and Theorem \ref{thm.Slutsky}, we obtain the following result.

\begin{lemma}\label{lem.Isom(UDelta).model.theory}
    Let $\Delta\subseteq\R_+$ be a distance set. Then $\Isom(\U_\Delta)$ has no algebraicity and admits weak elimination of imaginaries. 
\end{lemma}

\subsection{The Kat\v{e}tov construction and approximation}
We explain now how to build $\U_\Delta$ using the standard construction due to Kat{\v e}tov \cite{Katetov}. Let $\Delta$ be a distance set and $X$ be a countable $\Delta$-metric space. A $\Delta$-valued Kat{\v e}tov function on $X$ is a map $f \colon X\to \Delta$ satisfying 
\begin{align}
    \forall x,y\in X, \lvert f(x)-f(y)\rvert \leq d(x,y)\leq f(x)+f(y).
\end{align}
The Urysohn $\Delta$-metric space $\U_\Delta$ is characterized by the so-called \emph{Urysohn property}:

\begin{lemma}\label{Lem.Katetov}
    The Urysohn $\Delta$-metric space $\U_\Delta$ is, up to isometry, the unique countable $\Delta$-metric space $X$ with the $\Delta$-Urysohn property, i.e. such that for every finite subset $Y\subseteq X$ and every $\Delta$-valued Kat\v{e}tov function $f\colon Y\to\Delta$, there exists $x\in X$ such that $f=d(x,\cdot)$.
\end{lemma}  
Let $E_\Delta(X,\omega)$ be the set of $\Delta$-valued Kat{\v e}tov functions $f \colon X\to \Delta$ such that there exists a finite subset $F\subseteq X$ satisfying
\[\forall x\in X, f(x)=\underset{y\in F}{\min}(f(y)+d(x,y)).\]
Since $X$ is countable, so is $E_\Delta(X,\omega)$. We equip it with the uniform metric $d$ defined for $f,g\in E_\Delta(X,\omega)$ by $d(f,g)=\sup\{\lvert f(x)-g(x)\rvert, x\in X\}$. Then every isometry of $X$ extends uniquely to an isometry of $E(X,\omega)$ in such a way that the extension morphism $\Isom(X)\to\Isom(E(X,\omega))$ is an embedding of topological groups \cite[Prop.~2.5]{Melleray}. Here $\Isom(X)$ and $\Isom(E(X,\omega))$ are equipped with their respective permutation topology.

To construct $\U_\Delta$, start with the empty space $X_0$ and define inductively $X_{n+1}=E_{\Delta}(X_n,\omega)$. Identify $X_{n}$ as an isometric subspace of $X_{n+1}$ via the isometry $X\to E(X,\omega)$ given by $x\mapsto d(x,\cdot)$ and let $X=\bigcup_{n\geq 0}X_n$. Then $X$ is a countable $\Delta$-metric space which satisfies the Urysohn property by construction. So $X$ is the Urysohn $\Delta$-metric space $\U_\Delta$.

The following lemma is a direct consequence of the Kat{\v e}tov construction we just explained. It is also a straightforward application of a very general theorem due to Müller on Fraïssé structures with a stationary independence relation, see \cite{Muller}. 

\begin{lemma}\label{Lem.Ursyohn.embeddings}
    Let $\Delta,\Lambda\subseteq\R_+$ be two  distance sets such that $\Delta\subseteq\Lambda$. Then there exists an isometric embedding $f\colon \U_{\Delta}\to\U_{\Lambda}$ such that every isometry of $f(\U_\Delta)$ extends to an isometry of $\U_\Lambda$ in such a way that this extension yields an embedding $\Isom(\U_\Delta)\to \Isom(\U_\Lambda)$ of topological groups. 
\end{lemma}

We obtain the following approximation property for $\U_\Delta$. 

\begin{corollary}\label{cor.approximation}
    Let $\Delta$ be a distance set and $(\Delta_n)_{n\geq 0}$ an increasing sequence of distance sets such that $\bigcup_{n\geq 0}\Delta_n=\Delta$.
    Then there exists an approximating sequence $G_0\hookrightarrow G_1\hookrightarrow\dots\hookrightarrow\Isom(\U_\Delta)$ such that for every $n\geq 0$, the metric space $\Omega_n\subseteq\U_\Delta$ is isometric to $\U_{\Delta_n}$ and $G_n=\Isom(\Omega_n)$. 
\end{corollary}

\begin{proof}
For every $n\in \N$, let $\Omega_n = \U_{\Delta_n}$ and $G_n =\Isom(\U_{\Delta_n})$.  The Kat\v{e}tov construction allows us to see $\Omega_n$ as an increasing sequence of metric spaces with $\Omega\coloneqq\bigcup_{n\in\N}\Omega_n$ being (isometric to) $\U_\Delta$. Lemma \ref{Lem.Ursyohn.embeddings} gives the embeddings $G_0\hookrightarrow G_1\hookrightarrow\dots\hookrightarrow  \Isom(\U_\Delta)$.

It only remains to see that $\bigcup_{n\geq 0}G_n$ is a dense subgroup of $\Isom(\U_\Delta)$. To that aim, let $A\subseteq \U_\Delta$ be a finite subset and $g\in \Isom(\U_\Delta)$. It suffices to find $g' \in \bigcup_{n\geq 0}G_n$ such that $g'_{|A} = g_{|A}$. Since $A$ is finite, there exists $n\in \N$ such that $A\subseteq \Omega_n = \U_{\Delta_n}$. Then $g_{|A}$ is a partial isometry of $\U_{\Delta_n}$ which extends to an element $g' \in \Isom(\U_{\Delta_n})$ by ultrahomogeneity of Urysohn spaces. 
\end{proof}

We can now use the approximation property obtained in Corollary \ref{cor.approximation} to prove the following result. 

\begin{theorem}\label{thm.dissociation.Isom(U_Delta)}
    Let $\Delta$ be a distance set. Then every unitary representation of $\Isom(\U_\Delta)$ is dissociated. 
\end{theorem}

\begin{proof}
    Let $\Delta$ be a distance set. If $\Delta$ is finite, then $\Isom(\U_\Delta)$ is Roelcke precompact by Lemma \ref{lem.coarse.geometry}. Since its action on $\U_\Delta$ is transitive, $\Isom(\U_\Delta)$ is oligomorphic by \cite[Prop.~2.4]{Tsankov}. It moreover has no algebraicity and weakly eliminates imaginaries by Lemma \ref{lem.Isom(UDelta).model.theory}. Thus, Proposition 3.2 of \cite{JT} shows that every unitary representation of $\Isom(\U_\Delta)$ is dissociated. 
    
    Therefore, we may assume that $\Delta$ is countably infinite and fix an enumeration $\Delta=\{\delta_n\colon n\geq 0\}$ with $\delta_0=0$. For every $n\geq 0$, let $M_n\coloneqq\max\{\delta_0,\dots,\delta_n\}$. Let $\Delta_n$ be the intersection of the closed interval $[0,M_n]$ and the subsemigroup generated by $\{\delta_0,\dots,\delta_n\}$. Then $(\Delta_n)_{n\geq 0}$ forms an increasing sequence of finite distance sets. By Corollary \ref{cor.approximation}, we have an approximating sequence $$\Isom(\U_{\Delta_0})\hookrightarrow\Isom(\U_{\Delta_1})\hookrightarrow\dots\hookrightarrow\Isom(\U_\Delta).$$
    Since the $\Delta_n$'s are finite, every unitary representation of $\Isom(\U_{\Delta_n})$ is dissociated by the first case. By Theorem \ref{thm.dissociation.through.approximation}, we obtain that every unitary representation of $\Isom(\U_\Delta)$ is dissociated.  
\end{proof}

Combining with the results obtained in Section \ref{sec.dissociation.unirep}, we get the following. Elaborating on Remark \ref{rem.quotient}, note that for $G = \Isom(\U_\Delta)$ where $\Delta \subseteq \R_+$ is a distance set, $\Nm G A / G_A$ naturally identifies with $\Isom(A)$ for every finite subset $A \subseteq \U_\Delta$.

\begin{theorem}\label{thm.classification.unirep.Isom(UDelta)}
    Let $\Delta$ be a distance set and let $G=\Isom(\U_\Delta)$. Then every unitary representation of $G$ is a direct sum $\bigoplus_{i\in I}\Ind_{G_{\{A_i\}}}^G(\sigma_i)$, where $A_i$ are finite subsets of $\U_\Delta$ and $\sigma_i$ are irreducible unitary representations of the finite groups $\Isom(A_i)$.
\end{theorem}

\begin{remark}
    For a \emph{finite} distance set $\Delta$, we explained that $\Isom(\U_\Delta)$ is oligomorphic, has no algebraicity and eliminates weakly imaginaries.  Therefore, Theorem \ref{thm.dissociation.Isom(U_Delta)} and Theorem \ref{thm.classification.unirep.Isom(UDelta)} for $\Delta$ finite are special cases of results respectively due to the second author and Tsankov \cite[Prop.~3.2]{JT} and to Tsankov \cite[Cor.~5.2]{Tsankov}. For infinite distance set, the results are new.
\end{remark}

\section{Property (T) for $\Isom(\U_\Delta)$}\label{sec.property(T)}

The aim of this section is to use techniques developed by Tsankov in \cite{Tsankov} to prove property (T) for isometry groups of Urysohn $\Delta$-metric spaces. We start by recalling the definition of property (T) for topological groups. 

\begin{definition}
    A topological group $G$ has \defin{property (T)} if there exists a compact subset $Q\subseteq G$ and $\varepsilon>0$ such that every unitary representation $\pi \colon G\to\UU(\HH)$ with a non-zero vector $\xi\in\HH$ which is $(Q,\varepsilon)$-invariant in the sense that 
    \[\underset{g\in Q}{\sup\,}\lVert \pi(g)\xi-\xi\rVert \leq \varepsilon\lVert \xi\rVert,\]
     admits a non-zero invariant vector. The set $Q$ is called a Kazhdan set for $G$  and $\varepsilon$ a Kazhdan constant for $Q$. The couple $(Q,\varepsilon)$ is called a Kazhdan pair.
\end{definition}

To prove property (T) for $\Isom(\U_\Delta)$ we follow the strategies from Bekka \cite{Bekka} and Evans, Tsankov \cite{EvansTsankov}. For this, let us extract the following result from their works (see the proof of Theorem 2 in \cite{Bekka} and the proof of Theorem 1.1 in \cite{EvansTsankov}). 

\begin{theorem}\label{thm.property(T)}
    Let $G\leq\Sym(\Omega)$ be a closed permutation group. Assume that the following holds.
    \begin{itemize}
        \item Every irreducible unitary representation of $G$ is a subrepresentation of $\ell^2(G/V)$ for some open subgroup $V\leq G$.
        \item Every unitary representation of $G$ is a direct sum of irreducible ones.
        \item There exists a finite subset $Q\subseteq G$ generating a free group $F$, such that for all proper, open subgroup $V\leq G$, the action of $F$ on $G/V$ is free.
    \end{itemize}
     Then $G$ has property (T). More precisely, $(Q,\varepsilon)$ is a Kazhdan pair for G, with
    \[\varepsilon=\sqrt{2-\frac{2\sqrt{2\lvert S\rvert-1}}{\lvert S\rvert}}\]
    and $\varepsilon$ is the optimal Kazhdan constant for $Q$.
\end{theorem}

Notice that the first two items hold when $G=\Isom(\U_\Delta)$. Indeed, by Theorem \ref{thm.classification.unirep.Isom(UDelta)}, every unitary representation of $\Isom(\U_\Delta)$ is a direct sum of irreducible ones, which are of the form $\Ind_{\Nm G A}^G(\sigma)$ for some $A\subseteq\U_\Delta$ finite and some unitary representation $\sigma$ of $\Nm G A$ which factors through the finite group $F=\Nm G A/G_A$. But such a $\sigma$ is a subrepresentation of the left-regular representation $\lambda_F$ of $F$ and therefore we have
\begin{align*}
      \Ind_{G_{\{A\}}}^G(\sigma)\leq \Ind_{G_{\{A\}}}^G(\lambda_F) &\simeq \Ind_{G_{\{A\}}}^G(\Ind_{G_A}^{G_{\{A\}}}(1_{G_A}))  \\
      &\simeq \Ind_{G_A}^G(1_{G_A}) \\ &\simeq\ell^2(G/G_A).
  \end{align*}

Therefore, to prove Theorem \ref{thmintro.property(T)}, it remains to show the existence of a non-abelian free subgroup acting freely on every infinite transitive permutation representation of $\Isom(\U_\Delta)$. To do this, we first refer to the methods developed in \cite{FLMMM} to obtain the following:

\begin{lemma}\label{lem.free.group.freely}
    Let $\Delta\subseteq\R_+$ be a distance set. Then every countably infinite group admits a free isometric action on $\U_\Delta$. 
\end{lemma}

The authors of \cite{FLMMM} proved much stronger versions of the above statement for $\Delta$ bounded and $\Delta = \Q_+$ (Lemma 3.11 and Proposition 7.5 in \cite{FLMMM} respectively). However, their construction in the bounded case can easily be adapted in order to obtain the above statement. Indeed, say that an isometric action $\Gamma \curvearrowright (X,d)$ of a countable group $\Gamma$ on a metric space $(X,d)$ is \defin{strongly free} if for every $\gamma \neq e_\Gamma, \forall x \in X, \ d(\gamma\cdot x, x) \geqslant 1$. Given a distance set $\Delta\subseteq \R_+$ and a countable group $\Gamma$, start with the strongly free left-action $\Gamma \curvearrowright (\Gamma, d)$ where $d$ is the discrete metric ($d(\gamma, \gamma') =1$ if $\gamma \neq \gamma'$ and $0$ otherwise). Up to rescaling $d$, we can always assume that $(\Gamma,d)$ is a $\Delta$-metric space. The authors of \cite{FLMMM} defined a variation of the Kat\v etov construction such that at each step, the action of $\Gamma$ on the adapted Kat\v etov space remains strongly free. This is the content of Item 4 in \cite[Prop.~3.8]{FLMMM} and still holds when we remove the bound in the definition of the adapted Kat\v etov space. Applying this construction iteratively, we recover a strongly free action of $\Gamma$ on $\U_\Delta$, hence the above lemma.

Now, if one assumes that $\Gamma$ is torsion free, the action such obtained has the desired property:

\begin{corollary}\label{cor.torsion.free.acting.freely}
    Let $\Delta\subseteq\R_+$ be a distance set. Then for every torsion-free countable group $\Gamma$, there exists a free isometric action of $\Gamma$ on $\U_\Delta$ such that for all proper, open subgroup $V\leq\Isom(\U_\Delta)$, the action of $\Gamma$ on $\Isom(\U_\Delta)/V$ is free.
\end{corollary}

\begin{proof}
    Let $G=\Isom(\U_\Delta)$ and let $\Gamma$ be a torsion-free countable group. By Lemma \ref{lem.free.group.freely}, we fix $\Gamma\leq G$ whose action on $\U_\Delta$ is free. Let $V\leq G$ be a proper, open subgroup. Since $G$ has no algebraicity and weakly eliminates imaginaries by Lemma \ref{lem.NoAlg.WEI}, there exists a unique finite subset $A\subseteq\U_\Delta$ such that $G_A\leq V\leq G_{\{A\}}$. As $V$ is proper, then $A$ is non-empty. Assume that there exists a non-trivial element $\gamma\in\Gamma$ and a coset $gV\in G/V$ such that $\gamma\cdot gV=gV$. Then $\gamma\in gVg\inv$ and in particular, $\gamma$ fixes setwise $g(A)$. Since $\Gamma$ is torsion-free, some non-trivial power of $\gamma$ has a fixed point in $g(A)$. This contradicts the freeness of the action $\Gamma\curvearrowright\U_\Delta$. Thus, the action $\Gamma\curvearrowright G/V$ is free. 
\end{proof}

This applies in particular to non-abelian free groups. Therefore, $\Isom(\U_\Delta)$ satisfies every assumptions of Theorem \ref{thm.property(T)}. Thus, we have proved Theorem \ref{thmintro.property(T)}.

\section{A discussion on representations of $\Isom(\U)$}\label{sec.Urysohnbig}

Let $\U$ be the Urysohn space, which is the unique complete, separable metric space satisfying the following two conditions:
\begin{enumerate}
    \item (\emph{universality}) every separable metric space embeds isometrically into $\U$,
    \item (\emph{ultrahomogeneity}) every isometry between finite subspaces of $\U$ extends to an isometry of $\U$.
\end{enumerate}

One way of constructing $\U$ is by taking the completion of $\Q\U$. Let $\Isom(\U)$ be the group of all isometries of $\U$ onto itself. It is a Polish group when equipped with the topology of pointwise convergence (here $\U$ is equipped with the topology induced by its metric). The following result was obtained around a decade ago by Tent and Ziegler but never published. It can be deduced from their techniques developed in \cite{TentZiegler} and \cite{TentZieglerBounded}. 

\begin{theorem}[Tent, Ziegler, unpublished]
    The isometry group of the Urysohn space $\U$ is topologically simple.
\end{theorem}

The proof of the following lemma is straightforward.
\begin{lemma}\label{lem.trivial.continuous.homomorphism}
    Let $G,H,K$ be topological groups. Assume that $G$ is a closed subgroup of $H$, that $H$ is topologically simple and that every continuous homomorphism $G\to K$ is trivial. Then every continuous homomorphism $H\to K$ is trivial.  
\end{lemma}

We therefore obtain the following consequence on representations of $\Isom(\U)$. Recall that a Banach space $E$ is reflexive if the canonical evaluation map from $E$ to its bidual $E^{**}$ is a homeomorphism. If $E$ is a Banach space, let $\Isom(E)$ be the group of linear isometries of $E$, which we equip with the topology of pointwise convergence.

\begin{corollary}\label{cor.Isom(U)}
    The isometry group of the Urysohn space $\U$ admits no non-trivial continuous representation by isometry on a reflexive Banach space. 
\end{corollary}

\begin{proof}
    Let $\mathrm{Homeo}_+(\R)$ be the Polish group of orientation-preserving homeomorphisms of $\R$. By Uspenskij's theorem \cite{Uspenskij}, $\mathrm{Homeo}_+(\R)$ embeds into $\Isom(\U)$ as a topological group. By a result of Megrelishvili \cite{Megrelishvili}, for every reflexive Banach space $E$, every continuous homomorphism $\mathrm{Homeo}_+(\R)\to\Isom(E)$ is trivial. Then the same holds for $\Isom(\U)$ by Lemma \ref{lem.trivial.continuous.homomorphism}. 
\end{proof}

This answer a question of Pestov \cite{Pestov}, which follows exactly the same lines of proof to show that the isometry group of the Urysohn sphere has no non-trivial representation by isometry on a reflexive Banach space.

Let $(X,\mu)$ be a standard probability space and let $\Aut(X,[\mu])$ be the group of all non-singular bijections of $X$, i.e.\ the group of Borel bijections of $X$ which preserves $\mu$-null sets. The weak topology on $\Aut(X,[\mu])$ is the initial topology with respect to the family of functions $T\in\Aut(X,[\mu])\mapsto \mu(T(A)\triangle T(B))\in\R$ and $T\in\Aut(X,[\mu])\mapsto d(\mu\circ T)/d\mu\in \L^1(X,\mu)$. This turns $\Aut(X,[\mu])$ into a Polish group. A non-singular near-action of a topological group $G$ is a continuous group homomorphism $G\mapsto \Aut(X,[\mu])$. Using the standard Koopman representation $\Aut(X,[\mu])\to \UU(\L^2(X,\mu))$ defined by $T\mapsto (d(\mu\circ T\inv)/d\mu)^{1/2} f\circ T\inv$, we obtain the following result. 

\begin{corollary}
    The isometry group of the Urysohn space $\U$ admits no non-trivial non-singular near-action. 
\end{corollary}

We provide below an independent proof that $\Isom(\U)$ has no non-trivial unitary representation, where we use neither topological simplicity of $\Isom(\U)$, nor the result that $\mathrm{Homeo}_+(\R)$ has no non-trivial unitary representation.

\begin{proof}[Proof that $\Isom(\U)$ has no non-trivial unitary representation]
Let $G=\Isom(\U)$ and $\pi \colon G\to\UU(\HH)$ be a unitary representation. Let $H=\{g\in G\colon g(\Q\U)=\Q\U\}$. Notice that $H$ is a continuous homomorphic image of $\Isom(\Q\U)$, which is dense by Theorem 1 of \cite{CameronVershik}. By Theorem \ref{thmintro.unirep.Isom(UQ)}, $\pi_{|H}$ is a direct sum $\bigoplus_i\mathrm{Ind}_{H_{\{A_i\}}}^H(\sigma_i)$, where $A_i\subseteq\Q\U$ are finite subsets and $\sigma_i$ are irreducible representations of $H_{\{A_i\}}/H_A$.

\begin{fact}\label{fact.translation.UQ}
    There exists a sequence $(g_n)_{n\geq 0}$ of isometries of $\Q\U$ which converges to $\mathrm{id}_{\Q\U}$ for the topology of pointwise convergence (where $\Q\U$ is equipped with the topology arising from its metric), such that for every $n\geq 0$, the isometry $g_n$ fixes setwise no finite subset of $\Q\U$. 
\end{fact}
\begin{proof}[Proof of Fact \ref{fact.translation.UQ}]
    By a result of Cameron and Vershik \cite[Thm.~6]{CameronVershik}, there exists an isometry $g$ of $\Q\U$ such that $\langle g\rangle$ is transitive on $\Q\U$ and for every $h\in\langle g\rangle$, the displacement $d(x,h(x))$ is constant for $x\in\Q\U$. Fix $x\in\Q\U$ and a sequence $x_n$ of elements in $\Q\U\setminus\{x\}$ which converges to $x$. By transitivity, we fix for every $n\geq 0$ an element $k_n\in\Z$ such that $g^{k_n}(x_n)=x$ and set $g_n\coloneqq g^{k_n}$. Then $(g_n)_{n\geq 0}$ is as wanted. 
\end{proof}
\noindent Assume that $A_i$ is non-empty for some $i$. Denote by $\KK_i$ the Hilbert space of the representation $\sigma_i$ and by $\HH_i$ the Hilbert space of the representation $\mathrm{Ind}_{H_{\{A_i\}}}^H(\sigma_i)$. Fix a unit vector $\xi\in\KK_i$ and define $f\in\HH_i$ by
\[f(g)=\left\{\begin{array}{cl}
    \sigma(g)\xi & \text{if }g\in H_{\{A_i\}},  \\
    0 & \text{else}. 
\end{array}\right.\]
Since for every $n\geq 0$, the isometry $g_n$ doesn't fix $A_i$ setwise, we obtain that 
\[\lVert \pi(g_n)f-f\rVert =\sqrt{2}.\]
This is in contradiction with the continuity of $\pi_{|H}$ (where $H$ is here equipped with the topology of pointwise convergence). Hence $A_i$ is empty for every $i$ and $\pi$ is trivial on $H$. By density of $H$ in $G$, the representation $\pi$ is trivial. 
\end{proof}

\begin{remark}
    In \cite[Cor.~5.5]{Tsankov}, Tsankov used a similar strategy to prove that $\mathrm{Homeo}_+(\R)$ admits no continuous unitary representation. For this, he used the classification of unitary representations of $\Aut(\Q,<)$ he obtained. This is a special case of a more general result about representations by isometries of $\mathrm{Homeo}_+(\R)$ on Banach spaces due to Megrelishvili \cite{Megrelishvili}. 
\end{remark}

\printbibliography

\Addresses
\end{document}